\documentclass[a4paper]{article}
\usepackage[margin=2.3cm]{geometry}
\usepackage{latexsym}
\usepackage{indentfirst}
\usepackage{graphicx}
\usepackage{amsmath}
\usepackage{amssymb}
\usepackage{url}
\usepackage{csquotes}
\usepackage{amsthm}
\usepackage{parskip}

\newtheorem{definition}{Definition}
\newtheorem{example}{Example}

\newtheorem{remark}{Remark}
\newtheorem{theorem}{Theorem}
\newtheorem{prop}{Proposition}
\newtheorem{corollary}{Corollary}

\begin{document}

\title{Irreducible Bases and Subgroups of a Wreath Product in Applying to Diffeomorphism Groups Acting on the \text{M\"{o}bius} Band
}

\author{Ruslan V. Skuratovskii$^1$ and Aled Williams$^2$
\\ \\
$^1$Lecturer of IAPM, Kiev, Frometovskaya 2,\\ and Igor Sikorsky Kiev Polytechnic Institute, Kiev, \\ \hspace{1.0mm} \texttt{ruslcomp@mail.ru} \\ 
$^2$Cardiff University - \texttt{wiliamsae13@cardiff.ac.uk} \\
}

\date{}
\maketitle

\begin{abstract}
We generalize the
results presented in the book of Meldrum J. \cite{Meld} about commutator subgroup of wreath products since, as well as considering regular wreath
products, we consider those which are not regular (in the sense that the active group $\mathcal{ A}$ does not have to act
faithfully). The commutator of such a group, its minimal generating set and the centre of such products has been
investigated here.

The quotient group of the restricted and unrestricted wreath product by its commutator is found. The generic sets of commutator of wreath 
product were investigated.

The structure of wreath product with non-faithful group action is investigated.

Given a permutational wreath product sequence of cyclic groups, we investigate its minimal generating set, the minimal
generating set for its commutator and some properties of its commutator subgroup.

We strengthen the results from the author \cite{SkVC, SkAr, SkAr2, SkMal} and construct the minimal generating set for the wreath
product of both finite and infinite cyclic groups, in addition to the direct product of such groups.

The fundamental group of orbits of a Morse function $f:M\to \mathbb{R}$ defined upon a \text{M\"{o}bius} band $M$ with
respect to the right action of the group of diffeomorphisms $\mathcal{D}(M)$ has been investigated. In particular, we
describe the precise algebraic structure of the group $\pi_1 O(f)$. A minimal set of generators for the group of orbits of
the functions ${{\pi }_{1}}({{O}_{f}},f)$ arising under the action of the diffeomorphisms group stabilising the function
$f$ and stabilizing $\partial M$ have been found. The Morse function $f$ has critical sets with one saddle point.

We consider a new class of wreath-cyclic geometrical groups. The minimal generating set for this group and for the
commutator of the group are found.

This paper after previous Arxiv versions from 2019 \cite{SkArM, SkArM3} with previous title "Minimal generating set and structure of wreath product of groups with non-faithful action, comutator subgroup of wreath product and the fundamental group of orbit of Morse function $\pi_1 O(f)$" was published  \cite{SkRendi}.

\vspace{3.0mm}

\textbf{Acknowledgement}: We are grateful to Antonenko Alexandr for a graphical support and Sergey Maksymenko for Morse function description, also we thanks to Samoilovych I. 

\end{abstract}

\section{Introduction}
Lucchini A.
\cite{Luc} previously investigated a case of the generating set of $C_{p}^{n-1}\wr G$, where $G$ denotes a finite
$n$-generated group, $p$ is a prime which does not divide the order $|G|$ and $C_p$ denotes the cyclic group of order $p$.
The results of Lucchini A.
\cite{Luc} tell us that the wreath product $C_{p}^{n-1}\wr G$ is also $n$-generated.  We firstly consider the active group
$G$ which is cyclic and then generalize this wreath product for both iterated wreath products and for the direct product of
wreath products of cyclic groups. It should be noted that a similar question for iterated wreath product was studied by
Bondarenko I.
\cite{Bon}.


Maksymenko S.
\cite{Maks} has proven that the $n$-th homotopy groups of the orbit of $f$, $O(f)$, with respect to the right action of the
group, $\mathrm{Diff}(M)$, of diffeomorphisms of $M$, coincides with those of $M$ for $n\geq3$, i.e. $\pi_2 O(f)=0$, while,
for the fundamental group $\pi_1 O(f)$, it is known that it contains a free abelian subgroup of finite index. Despite this,
information regarding the fundamental group $\pi_1 O(f)$ remains incomplete. We provide some insight by finding the minimal
generating set and its relations for the group $\pi_1 O(f)$.

All theorems and propositions are obtained and proved by the Ruslan Skuratovskii, corollaries and examples were obtained in collaboration with the co-author.

This paper after previous Arxiv versions \cite{SkArM, SkArM3} with previous title "Minimal generating set and structure of wreath product of groups with non-faithful action, comutator subgroup of wreath product and the fundamental group of orbit of Morse function $\pi_1 O(f)$" was published  \cite{SkRendi}.

\section{Preliminaries}
Let $G$ be a group and let $d(G)$ denote its minimal number of generators \cite{Luc, Bon}.
A diffeomorphism $h : M \to M$
is said to be $f$-preserving if $f \circ h = f$. This is equivalent to the assumption
that $h$ is invariant each level-set, i.e. $f^{-1}(c)$, $c\in P$ of $f$, where $P$ denotes either the real line $R$ or the
circle $S^1$.

The \textit{commutator width} of $G$ \cite{nikolov}, denoted by $cw(G)$, is defined to be the least
integer $n$, such that every element of $G'$ is a product of at most $n$ commutators if such an integer exists, and
otherwise is $cw(G) = \infty$.
The first example of a finite perfect group with $cw(G) > 1$ was presented by Isaacs I. \cite{Isacs}. The property of
commutator widths for groups and elements has proven to be important and in particular, its connections with stable
commutator length and bounded cohomology has become significant.

Meldrum J. \cite{Meld} briefly considered one form of commutators of the wreath product $A \wr B$. In order to obtain a
more detailed description of this form, we take into account the commutator width by $(cw(G))$ as presented in work of Muranov
A. \cite {Mur}. 
${S_{{2^{k}}}})$ and $cw (C_p \wr B)$ \cite{SkComm}.

The form of commutator presentation \cite{Meld} has been given here in the form of wreath recursion \cite{Lav} and
additionally, its commutator width has been studied.

The subtree of $X^{*}$ (or $\mathbb{T}$) which is induced by the set of vertices $\cup_{i=0}^k X^i$ is denoted by
$X^{[k]}$ or $\mathbb{T}_k$. Denote the restriction of the action of an automorphism $g\in AutX^*$ to the subtree $X^{[l]}$ by
$g_{(v)}|_{X^{[l]}}$. It should be noted that a restriction $g_{(v)}|_{X^{[1]}} $ is called the \textit{vertex permutation}
(v.p) of $g$ in a vertex $v$.

\section{Commutator subgroup and center of wreath product with non-faithful action}
 In this work, we strengthen and continue the previous results of the author \cite{SkVC} and will additionally consider a new class of groups. This
class is precisely the \emph{wreath-cyclic} groups and will be denoted by $\Im $. Let $G \in \Im $, then this class is
constructed by formula:
$$G=(\underset{j_0=0}{\overset{n_0}{\mathop{\wr }}}{{C}_{k_{j_0}}})\times (\underset{j_1=0}{\overset{n_1}{\mathop{\wr
}}}{C_{k_{j_1}}})\times \cdots \times(\underset{j_l=0}{\overset{n_l}{\mathop{\wr }}}{{C}_{{{k}_{j_l}}}}), 1 \leq k_{j_i} <
\infty, n_i<\infty,$$
where the orders of ${{C}_{{{i}_{j}}}}$ are denoted by ${{i}_{j}}$.

It should be noted that at the end of this product, a semidirect product could arise with a given homomorphism $\phi $,
which is defined by a free action on the set $\mathbb{Z}$. In other words, one would obtain a group of the form ${{\left(
\prod\limits_{i=1}^{k}{{{G}_{i}}} \right)}^{n}}{{\ltimes }_{\phi }}\mathbb{Z}.$

Note that the last group here is isomorphic to one of the fundamental orbital groups ${{O}_{f}}(f)$ of the Morse function
$f$. Namely, we have ${{\pi }_{0}}\left( S,f\left| _{\partial M} \right. \right)$ \cite{Maks}.

Consider now the group $H=\underset{j=1}{\overset{n}{\mathop{\wr }}}{{C}_{{{i}_{j}}}}$, whose orders ${{i}_{j}}$ for all
${{C}_{{{i}_{j}}}}$ are mutually coprime for all $j>1$ and whose number of cyclic factors in the wreath product is finite.
We will call such group $H$ \emph{wreath-cyclic}.

Note that the multiplication rule of automorphisms $g$, $h$ which are presented in the form of wreath recursion \cite{Ne}
$g=(g_{(1)},g_{(2)},\ldots,g_{(d)})\sigma_g, \
h=(h_{(1)},h_{(2)},\ldots,h_{(d)})\sigma_h,$ is given precisely by the formula:
$$g\cdot h=(g_{1}h_{\sigma_g(1)},g_{2}h_{\sigma_g(2)},\ldots,g_{d}h_{\sigma_g(d)})\sigma_g \sigma_h.$$

In the general case, if an active group is not cyclic, then the cycle decomposition of an $n$-tuple for automorphism
sections will induce the corresponding decomposition of the $\sigma_g$.
If $\sigma$ is v.p of automorphism $g$ at $v_{ij}$ and all the vertex permutations below $v_{ij}$ are trivial, then we do
not distinguish $\sigma$ from the section $g_{v_{ij}}$ of $g$ which is defined by it. That is to say, we can write
$g_{v_{ij}}=\sigma = (v_{ij}) g$ as proposed by Bartholdi L., Grigorchuk R. and {\v{S}}uni Z. \cite{Grig}.

\subsection{Minimal generating set of direct product of wreath product of cyclic groups}
We now make use of both rooted and directed automorphisms as introduced by Bartholdi L., Grigorchuk R. and {\v{S}}uni Z.
\cite{Grig}. Recall that we denote a truncated tree by $\mathbb{T}$.

\begin{definition}
An automorphism of 
$\mathbb{T}$ is said to be rooted if all of its vertex permutations corresponding
to non-empty words are trivial.
\end{definition}

Let $l = x_1x_2x_3 \cdots $ be an infinite ray in $\mathbb{T}$.

\vspace{2.0mm}

\begin{definition}[]
The automorphism $g$ of $\mathbb{T}$ is said to be directed along the infinite ray $l$ if all vertex permutations along $l$
and all vertex permutations corresponding to vertices whose distance to the ray $l$ is at least two are trivial. In such
case, we say that $l$ is the spine of $g$ (as exemplified in Figure 1).
ray $l$ is at least two are trivial, then we say that the automorphism $g$ of $\mathbb{T}$ is directed along $l$ and we
will call $l$ the spine of $g$ (Figure 1).
\end{definition}

It should be noted that because we consider only truncated trees and truncated automorphisms here and for convenience, we
will say rooted automorphism instead of truncated rooted automorphism.


\begin{figure}[h]
\begin{minipage}{0.5\linewidth}
{\includegraphics[scale=0.5]{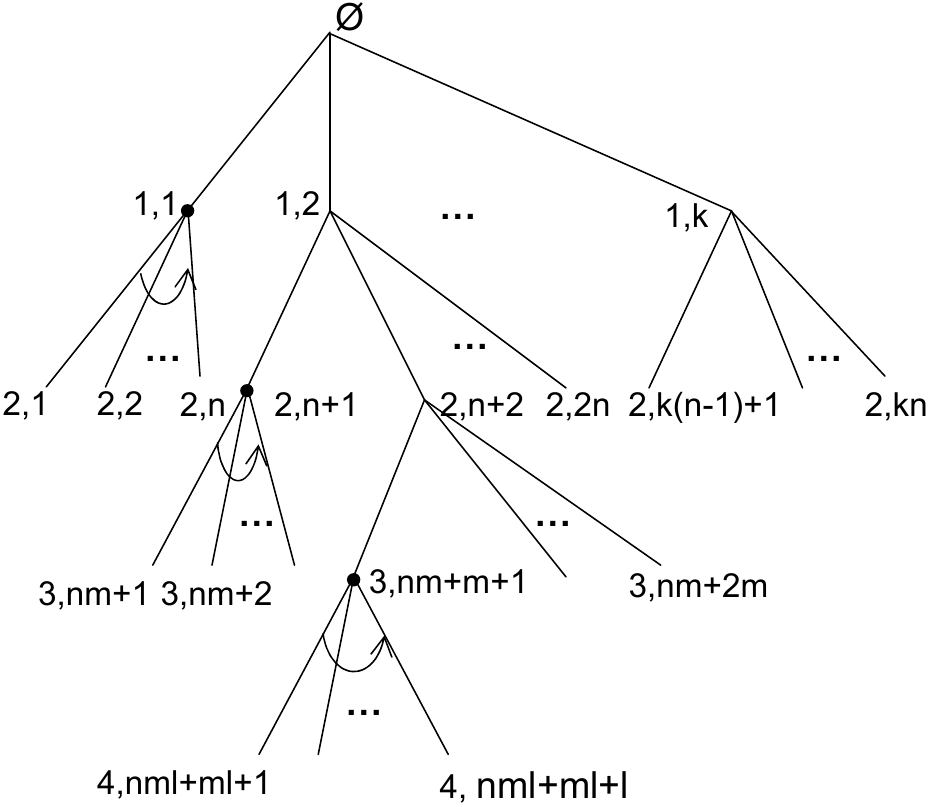}\\ ~\\}
\end{minipage}
\hfill
\begin{minipage}{0.5\linewidth}
{\includegraphics[scale=0.5]{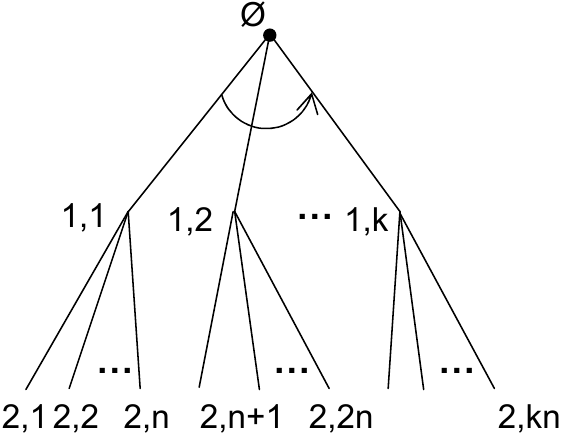}\\ ~\\}
\end{minipage}

\begin{minipage}{0.5\linewidth}
{Fig. 1. Directed automorphism}
\end{minipage}
\hfill
\begin{minipage}{0.5\linewidth}
{Fig. 2. Rooted automorphism}
\end{minipage}
\label{ris:image1}
\end{figure}


We recall in reformulated form the result of  A. Woryna, \cite{Wor} about a minimal generating set of iterated wreath product. Also we generalised this result after this theorem.
\begin{theorem}
If orders of cyclic groups ${{\mathbb{C}}_{{{n}_{i}}}}$, ${{\mathbb{C}}_{{{n}_{j}}}}$ are mutually coprime $i\ne j$, then
the group $G= {{C}_{{{i}_{1}}}} \wr {{C}_{{{i}_{2}}}} \wr \cdots \wr {{C}_{{{i}_{m}}}} $ admits two generators, namely
$\beta _{0}$, ${\beta }_{1}$.
\end{theorem}

\begin{proof}
Construct the generators of $\underset{j=0}{\overset{n}{\mathop{\wr }}}{{C}_{{{i}_{j}}}}$ as a rooted automorphism ${{\beta
}_{0}}$ (Figure 2) and a directed automorphism ${{\beta }_{1}}$ \cite{Grig} along a path $l$ (Figure 1) on a rooted labeled
truncated tree ${{T}_{X}}$.

We consider the group $G={{C}_{{{i}_{1}}}} \wr {{C}_{{{i}_{2}}}} \wr \cdots \wr {{C}_{{{i}_{m}}}}$.
Construct the generating set of ${{C}_{{{i}_{1}}}} \wr {{C}_{{{i}_{2}}}} \wr \cdots \wr {{C}_{{{i}_{m}}}}$, where the
active group is on the left.
Denote by $lc{{m}_{1}}=lcm({{i}_{2}},{{i}_{3}},\ldots,{{i}_{m}})$ the least common multiplier of the orders by
${{i}_{2}},{{i}_{3}},\ldots,{{i}_{m}}$. In a similar fashion, we denote
$lc{{m}_{k}}=lcm({{i}_{1}},{{i}_{2}},\ldots,{{i}_{k-1}},{{i}_{k+1}},\ldots,{{i}_{m}})$ similarly.

We utilise a presentation of those wreath product elements from a tableaux of Kaloujnine L. \cite{Kal} which has the form
$\sigma =\left[ {{a}_{1}},{{a}_{2}}(x),\,\,{{a}_{3}}\left( {{x}_{1}},{{x}_{2}} \right),\ldots \right]$. Additionally, we
use a subgroup of tableau with length $n$ which has the form ${{\sigma }_{(n)}}=\left[
{{a}_{1}},{{a}_{2}}({{x}_{1}}),\ldots{{a}_{n}}({{x}_{1}},\ldots,{{x}_{n}}) \right]$. The tableaux which has first $n$
trivial coordinates was denoted in \cite{ShNorm} by
$$^{(n)}\sigma =\left[ e,\ldots,e,{{\alpha }_{n+1}}({{x}_{1}},\ldots,{{x}_{n}}),{{\alpha
}_{n+1}}({{x}_{1}},\ldots,{{x}_{n+1}}),\ldots \right].$$
The canonical set of generators for the wreath product of ${{C}_{p}}\wr \cdots \wr {{C}_{p}}\wr {{C}_{p}}$ was used by
Dmitruk Y. and Sushchanskii V. \cite{Dm} and additionally utilised by the author \cite{SkKib}. This set has form
\begin{equation} \label{can}
{{\sigma' }_{1}}=\left[ {{\pi }_{1}},\,\,e,\,\,\,e,\ldots,e \right],  {{\sigma'}_{2}}=\left[ {{e}_{1}},\,\,{{\pi
}_{2}},\,\,\,e,\ldots,e \right], \ldots ,  {{\sigma'}_{n}}=\left[ {{e}_{1}},\,\,e,\ldots,e,{{\pi }_{n}} \right].
\end{equation}

We split such a table into sections with respect to (\ref{can}), where the $i$-th section corresponds to portrait of
$\alpha$ at $i$-th level. The first section corresponds to an active group and the crown of wreath product $G$, the second
section is separated with a semicolon to a base of the wreath product.
 The sections of the base of wreath product are divided into parts by semicolon and these parts correspond  to  groups
 ${{C}_{{{i}_{j}}}}$ which form the base of wreath product. The $l$-th section of of a tableau presentation of automorphism
 $\beta_1$ corresponds to portrait of automorphism $\beta_1$ on level $X^{l}$.

The portrait of automorphisms $\beta_1$ on level $X^{l}$ is characterised by the sequence $(e, \ldots ,e,
\pi_{l},e,\ldots,e)$, where coordinate $\pi_{l}$ is the vertex number of unique non trivial v.p on $X^l$, the sequence has
${i_{0}}{i_{1}} \ldots {i_{l-1}}$ coordinates.
Therefore, our first generator has the form
${{\beta }_{0}}=\left[ {{\pi }_{1}},e,e,\ldots,e \right]$, which is the rooted automorphism. The second generator has the
form
$${{\beta }_{1}}=\left[ e;\underbrace{{{\pi
}_{2}},\,e,e,\ldots,e}_{{{i}_{1}}};\overbrace{\underbrace{e,e,\ldots,e}_{{{i}_{2}}},{{\pi
}_{3}},e,\ldots,e}^{{{i}_{1}}{{i}_{2}}};\underbrace{\overbrace{e,\ldots,e,}^{{{i}_{2}}{{i}_{3}}+i_3}{{\pi
}_{4}},e,\ldots,e}_{{{i}_{1}}{{i}_{2}}{{i}_{3}}};\,e, \,\ldots\,\, ,e \right], $$

It should be noted that after the last(fourth) semicolon (or in other words before $\pi_5$) there are
$i_2i_3i_4+i_3i_4+i_4$ trivial coordinates. There are $i_2i_3i_4i_5 + i_3i_4i_5 +i_4 i_5+i_5$ trivial coordinates before
$\pi_6$ (or in other words after the fifth semicolon but before $\pi_6$). In a section after $k-1$ semicolon the coordinate
of a non-trivial element $\pi_k$ is calculated in a similar way.
We know from \cite{ShNorm} that ${{\beta }_{1}}$is generator of $^{(2)}G$, i.e. 2-base of $G$. Recall that $^{(k)}G$ calls
\emph{k}-th base of $G$. The subgroup $^{(k)}G$ is a subgroup of all tableaux of form $^{(k)}u$ with $u\in G$.

\vspace{5.0mm}

Let ${{C}_{n}}=\left\langle {{\pi }_{n}} \right\rangle $ and set $\sigma_1 =  {\beta }_{0}$. We have to show that our
generating set $\{ {\beta }_{0}, {\beta }_{1} \}$ generates the whole canonical generating set.
For this, we obtain the second new generator ${\sigma }_{2}$ in form of the tableau
  $${\sigma }_{2}^{lc{{m}_{2}}}=\beta _{1}^{lc{{m}_{2}}}=\left[ e;\underbrace{\pi
  _{2}^{lc{{m}_{2}}},\,e,e,\ldots,e}_{{i}_{1}};\underbrace{e,e,\ldots,e}_{{{i}_{1}}{{i}_{2}}};
  \underbrace{e,e,\ldots,e}_{{{i}_{1}}{{i}_{2}}{{i}_{3}}};\,e,\,\ldots\,,\,e \right]. $$

Since $\text{ord}({{\pi }_{1}})={{i}_{1}}$ and $({{i}_{1}},\,\,lc{{m}_{1}})=1$, we find that the element $\pi
_{1}^{lc{{m}_{1}}}$ is generator of ${{C}_{{{i}_{1}}}}$ since $\text{ord}({{\pi }_{1}})=\text{ord}(\pi
_{1}^{lc{{m}_{1}}})$. We obtain that $${{\sigma }_{2}}= {{\left( \beta _{1}^{lc{{m}_{2}}}
\right)}^{lcm_{2}^{-1}(mod\,{{i}_{2}})}},$$
which corresponds to generator ${{\sigma }_{2}}$ of canonical generating set (\ref{can}).
Observe that ${{b}_{3}}=\sigma _{1}^{-1}\beta _{1}$ is generator of $^{(3)}G$, i.e. it is precisely a 3-base of $G$.


\vspace{5.0mm}

It is known \cite{ShNorm} that the generator ${{\sigma}_{2}}$ precisely generates the group that is isomorphic to the group
${{\left[ U \right]}_{2}}$  for all $2$-nd coordinate tableaux. From the same principle, one can obtain that
\[{{\sigma }_{3}}=\beta _{1}^{lc{{m}_{3}}}=\left[
e;\underbrace{e,\,e,e,\ldots,e}_{{{i}_{1}}};\overbrace{\underbrace{e,e,\ldots,e}_{{{i}_{2}}},\pi
_{3}^{lc{{m}_{3}}},e,\ldots,e}^{{{i}_{1}}{{i}_{2}}};\underbrace{\overbrace{e,\ldots,e,}^{{{i}_{2}}{{i}_{3}}}e,e,\ldots,e}_{{{i}_{1}}{{i}_{2}}{{i}_{3}}};\,e\,\ldots\,\,e
\right].\]
This  generator ${{\sigma }_{3}}$  generates the group which is isomorphic to the group of all $(2i_1+2)$-th coordinate
tableaux, which is precisely ${{\left[ U \right]}_{2i_1+2}}$ \cite{ShNorm}.
Making use of the same principle allows us to express all the ${{\sigma }_{i}}$ from our canonical generating set.

Note that if it were a self-similar group, then it would be more useful to present it in terms of wreath recursion, as the
set where ${\beta_0 }$ is the rooted automorphism. Given a permutational representation
of $C_{i_j}$ we can present our group by wreath recursion.
 We present ${\beta_1}$ by wreath recursion as ${\beta_1 } = ({{\pi}_{2}}, {{\beta }_{2}},e,e, \ldots ,e)$.
 It would be written in form $\sigma _{1}^{lcm_{2}}={{\beta_1 }^{lc{{m}_{2}}}}=({{\pi }_{2}}^{lc{{m}_{2}}},{{\beta
 }_{2}}^{lc{{m}_{2}}},e,e,\ldots,e)=({{\pi }_{2}}^{lcm(2)},e,e,\ldots,e)$, since $\text{ord}({{\pi }_{2}})={{i}_{2}}$ and
 $({{i}_{2}},\,\,lc{{m}_{2}})=1$  then the element $\pi _{2}^{lcm_{2}}$ is generator of ${{C}_{{{i}_{2}}}}$ too, because
 $\text{ord}({{\pi }_{2}})=\text{ord}(\pi _{2}^{lcm_{2}})$.

We then obtain the second generator ${{\sigma }_{2}}$ of canonical generating set by exponentiation ${{\left( \beta
_{1}^{lc{{m}_{2}}} \right)}^{lcm_{2}^{-1}(mod\,{{i}_{2}})}}=\left( {{\pi }_{2}},e,\ldots,e \right)$.
Since we have obtained ${{\sigma }_{2}}=\left( {{\pi }_{2}},e,\ldots,e \right)$, we can express $\sigma _{2}^{-1}=\left(
\pi _{2}^{-1},e,\ldots,e \right)$, where ${{\pi }_{2}}$ is a state of $\sigma_{2}$.

Consider an alternative recursive constructed generating set which consists of nested automorphism $\beta_1$ states which
are $\beta_2$, $\beta_3$,\ldots,$\beta_m$ and the automorphism $\beta_0$.
The state $\beta_2$ is expressed as follows ${\sigma }_{2}^{-1}{\beta }_{1}=\left( e, {{\beta }_{2}},e,\ldots,e \right)$.

It should be noted that a second generator of a recursive generating set could be constructed in an other way, namely
${{\beta' }_{2}}={{\beta_1}^{{{i}_{2}}}}=({{\pi }_{2}}^{{{i}_{2}}},{{\beta }_{2}}^{{{i}_{2}}},e,e,\ldots,e)=(e,{{\beta
}_{2}}^{{{i}_{2}}},\ldots,e,e)$, where ${\beta_2 }$ is the state in a vertex of the second level $X^2$.

We can then express the next state ${\beta }_{2}$ of ${\beta }_{1}$ by multiplying $\sigma _{2}^{-1}{{\beta }_{1}} = (
e,{{\beta }_{2}},e,\ldots,e )$. Therefore, by a recursive approach, we obtain ${{\beta }_{2}}= ( {{\pi }_{3}},{{\beta
}_{3}},e,\ldots,e )$ and analogously we obtain $\beta _{2}^{lc{{m}_{3}}}=\sigma _{3}^{lc{{m}_{3}}}= ( \pi
_{3}^{lc{{m}_{3}}},e,\ldots,e )$.
Similarly, we obtain $$\beta _{k-1}^{lc{{m}_{k}}}={{\sigma }^{lcm_{k}}_{k}}=\left( \pi _{k}^{lc{{m}_{k}}},e,\ldots,e
\right)$$ and ${{\sigma }_{k}}={{\left( \beta _{k-1}^{lc{{m}_{k}}} \right)}^{lcm_{k}^{-1}(\text{mod } {{i}_{k}})}}= ( {{\pi
}_{k}},e,\ldots,e )$. The $k$-th generator of the recursive generating set can therefore be expressed as ${\sigma
}_{k}^{-1}{\beta }_{k-1}=\left( e, {{\beta }_{k}},e,\ldots,e \right)$.
The last generator of our generating set has another structure, namely ${{\sigma }_{m}}=(\pi_{m} , e, \ldots, e )$ which
concludes the proof.
\end{proof}

\vspace{2.0mm}








Let $\underset{j=0}{\overset{n}{\mathop{\wr }}}{{C}_{{{i}_{j}}}}$ be generated by ${\beta }_{0}$ and ${\beta }_{1}$ and $
\underset{l=0}{\overset{m}{\mathop{\wr }}}{{C}_{{{k}_{l}}}}= \langle{ {\alpha }_{0}}, {{\alpha }_{1}} \rangle $. Denote an
order of $g$ by $|g|$.

\vspace{5.0mm}

\begin{theorem}
If $\left( |{{\alpha }_{0}}|,|{{\beta }_{0}}| \right)=1$ and $\left( |{{\alpha }_{1}}|,|{{\beta }_{1}}| \right)=1$ or
$\left( |{{\alpha }_{0}}|,|{{\beta }_{1}}| \right)=1$ and $\left( |{{\alpha }_{1}}|,|{{\beta }_{0}}| \right)=1$, then there
exists generating set of 2 elements for the wreath-cyclic group $$G=(\underset{j=0}{\overset{n}{\mathop{\wr
}}}{{C}_{{{i}_{j}}}})\times (\underset{l=0}{\overset{m}{\mathop{\wr }}}{{C}_{{{k}_{l}}}}),$$ where ${{i}_{j}}$ are orders
of ${{C}_{{{i}_{j}}}}$.
\end{theorem}

\begin{proof}
The generators ${{\alpha }_{1}}$ and ${{\beta }_{1}}$  are directed automorphisms, ${{\alpha }_{0}},\,\,{{\beta }_{0}}$
are rooted automorphisms \cite{Grig}.
The structure of tableaux are described above in Theorem 1.
In case $\left( |{{\alpha }_{0}}|,|{{\beta }_{0}}| \right)=1$ are mutually coprime and $\left( |{{\alpha }_{1}}|,|{{\beta
}_{1}}| \right)=1$ are mutually coprime, then we group  generator  ${{\alpha }_{0}}$  and  ${{\beta }_{0}}$ in  vector that
is first generator of direct product $(\underset{j=0}{\overset{n}{\mathop{\wr }}}{{C}_{{{i}_{j}}}})\times
(\underset{l=0}{\overset{m}{\mathop{\wr }}}{{C}_{{{k}_{l}}}})$.
Therefore, the first generator of $G$ has form $({{\alpha }_{0}},\,\,{{\beta }_{0}})$ and the second generator has form of
vector $({{\beta }_{1}},\,\,{{\alpha }_{1}}\,)$.
The generator ${{\alpha }_{1}}$ has a similar structure.

In order to express the generator $\sigma_2$ of the canonical set (\ref{can}) from $
\langle {{\alpha }_{0}}, {{\beta }_{1}} \rangle$ we change the exponent from ${\beta }_{1}$ to $lcm_2$. Analogously, we
obtain $\sigma_k = {{\beta }_{1}}^{lcm_k}$ which concludes the proof.
\end{proof}


\subsection{Generators of commutator subgroup and center of a wreath product of group with a non-faithful action }
Let us find upper bound of generators number for $G'$.
Let $ \mathcal{A}$ be a group and $\mathcal{B}$ a permutation group, i.e. a group $\mathcal{A}$ acting upon a set $X$,
where the active group $\mathcal{ A}$ can act not faithfully. Consider the set of all pairs $\{(a, f),  f: X \rightarrow h,
a\in \mathcal{A} \}$.	
We define a product on this set as
$$\{(a_1, f_1) (a_2, f_2):= (a_1 a_2,  f_1 f_2 ^{a_1})\},$$
where $f_1^{a_2}(x)= f_1({a_2}(x))$.

\vspace{5.0mm}

\begin{theorem}
If $W=(\mathcal{A},X)\wr (\mathcal{B},Y)$, where $\left| X \right|=n,\,\,\left| Y \right|=m$ and active group
$\mathcal{A}$ acts on $X$ transitively, then $$d\left( G' \right)\le (n-1)d(\mathcal{B})+d(\mathcal{B}')
+d(\mathcal{A}').$$
\end{theorem}

\begin{proof}
The generators of $W'$ in form of tableaux \cite{Agn}: ${{a'}_{i}}=({{a}_{i}};e,e,e, \ldots, e)$,
${{t}_{1}}=(e;{{h}_{{{j}_{1}}}},e,e, \ldots,{{c}_{{{j}_{1}}}}), \ldots , \,\,{{t}_{k}}=(e;e,e,e, \ldots,
{{h}_{{{j}_{k}}}},e, \ldots, {{c}_{{{j}_{k}}}}),$ ${{t}_{{{l}}}}=(e;e,e,e, \ldots, {{h}_{{{j}_{l}}}}, {{c}_{{{j}_{l}}}})$,
where ${{h}_{j}}, c_{{j}_{l}} \in {{S}_{B}}$, $\mathcal{B}=\left\langle {{S}_{B}} \right\rangle $, ${{a}_{i}}\in
{{S}_{A}}$, $A=\left\langle {{S}_{A}} \right\rangle $.
	Note that, on a each coordinate of tableau, that presents a commutator of  $[a;{{h}_{\text{1}}},\ldots,{{h}_{n}}]$ and
$[b;{{g}_{1}},\ldots,{{g}_{n}}]$, $a, \,b \in \mathcal{A}, h_i, g_j\in \mathcal{B}$ can be product of form
${{a}_{1}}{{a}_{2}}a_{1}^{-1}a_{2}^{-1}\in \mathcal{A}'$ and
${{h}_{i}}{{g}_{a(i)}}h_{ab(i)}^{-1}g_{ab{{a}^{-1}}(i)}^{-1}\in \mathcal{B}$, according to Corollary 4.9  \cite{Meld}.
This products should satisfy the following condition:

	\begin{equation} \label{CommCon}
		\prod\limits_{i\in {{X}}}^{n}{{{h}_{i}}{{g}_{a(i)}}h_{ab(i)}^{-1}g_{ab{{a}^{-1}}(i)}^{-1}\in \mathcal{B}'}.
	\end{equation}
		That is to say that the product of coordinates of wreath product base is an element of commutator $\mathcal{B}'$.
	As it was described above it is subdirect product of $\underbrace{\mathcal{B}\times \mathcal{B}\times \cdots \times
\mathcal{B}}_{n}$ with the additional condition (\ref{CommCon}).
	This is the case because not all element of the subdirect product are independent because the elements must be chosen
in such a way that (\ref{CommCon}) holds. We may rearrange the factors in the product in the following way:

\[\prod\limits_{i=1}^{n}{{{h}_{i}}{{g}_{a(i)}}h_{ab(i)}^{-1}g_{ab{{a}^{-1}}(i)}^{-1}}=(\prod\limits_{i=1}^{n}{{{h}_{i}}{{g}_{{{i}_{{}}}}}h_{{{i}_{{}}}}^{-1}g_{{{i}_{{}}}}^{-1}})[g,h]\in
\mathcal{B}'.\]
where $[g,h]$ is a commutator in case $cw(B)=1$.  We express this element from $\mathcal{B}'$  as commutator $[g,h]$ if
$cw(B)=1$. In the general case, we would have $\prod\limits_{j=1}^{cw(B)}{[{{g}_{j}},{{h}_{j}}]}$ instead of this element.
This commutator are formed as product of commutators of rearranged elements of
$\prod\limits_{i=1}^{n}{{{h}_{i}}{{g}_{a(i)}}h_{ab(i)}^{-1}g_{ab{{a}^{-1}}(i)}^{-1}}$.
  Therefore, we have a subdirect product of $n$ the copies of the group $B$ which has been equipped by condition
  (\ref{CommCon}).
  The multiplier $\prod\limits_{j=1}^{cw(B)}{[{{g}_{j}},{{h}_{j}}]}$ from $\mathcal{B}'$, which has at least
  $d(\mathcal{B}')$ generators
		
\begin{equation}\label{Cent0}
\prod\limits_{i=1}^{n}{{{h}_{i}}{{g}_{a(i)}}h_{ab(i)}^{-1}g_{ab{{a}^{-1}}(i)}^{-1}}=(\prod\limits_{i=1}^{n}{{{h}_{i}}{{g}_{{{i}_{{}}}}}h_{{{i}_{{}}}}^{-1}g_{{{i}_{{}}}}^{-1}})\prod\limits_{j=1}^{cw(B)}{[{{g}_{j}},{{h}_{j}}]}\in
\mathcal{B}'.
\end{equation}

	Since $(\prod\limits_{i=1}^{n}{{{h}_{i}}{{g}_{{{i}_{{}}}}}h_{{{i}_{{}}}}^{-1}g_{{{i}_{{}}}}^{-1}})=e$ and  the product
$\prod\limits_{j=1}^{cw(B)}{[{{g}_{j}},{{h}_{j}}]}$ belongs to $\mathcal{B}'$,  then condition (\ref{CommCon}) holds.  The
assertion of a theorem on a recursive principle is easily generalized on multiple wreath product of groups.
	
	Thus minimal total amount consists of at least $d\left( \mathcal{B}' \right)$ generators for $n-1$ factors of group
$\mathcal{B}$, $d\left( \mathcal{B}' \right)$ generators for the dependent factor from $\mathcal{B}'$ and $d(\mathcal{A})$
generators of the group $\mathcal{A}$ which concludes the proof.

It should be noted that not all the elements of commutator subgroup, that has structure of the subdirect product, are
independent by (\ref{CommCon}), at least one of them must be chosen carefully such that would be (\ref{CommCon}) satisfied.
This implies the estimation $d\left( G' \right)\le (n-1)d(B)+d(B')$.

	Thus minimal total amount consists of at least $d\left( \mathcal{B}' \right)$ generators for $n-1$ factors of group
$\mathcal{B}$, $d\left( \mathcal{B}' \right)$ generators for the dependent factor from $\mathcal{B}'$ and $d(\mathcal{A})$
generators of the group $\mathcal{A}$ which concludes the proof.
\end{proof}

\vspace{5.0mm}

We shall consider special case when a passive group $(\mathcal{B},Y)$ of $W$ is a perfect group.
Since we obtain a direct product of $n-1$ the copies of the group $B$ then according to Corollary 3.2. of Wiegold J.
\cite{W} $d\left( \mathcal{B}^n  \right) \le d(\mathcal{B}) + n-1 $ \cite{W}. More exact upper bound give us Theorem A.
\cite{W}, which use $s$ a the size of the
smallest simple image of $G$.
	 As it were studied by J. Wiegold if among group of product $\prod\nolimits_{}{}_{i=1}^{n}({{G}_{i}},{{X}_{i}})$ are
$n$ same perfect groups   then $d({{({{G}_{0}},{{X}_{0}})}^{n}})<d({{G}_{\text{0}}})+1+{{\log }_{s}}n$, where $s=ord(H)$
where $H$ is minimal prime non abelian group, which is image of  ${{G}_{0}}$  \cite{W}.
Therefore, in this case our upper bound has the form
  $$d\left( W' \right)\le clog_s n + d(\mathcal{B}') + d(\mathcal{A}').$$

Now we consider no regular wreath product, where active group can be both as infinite as finite and consider a centre of
such group. This is generalization of Theorem 4.2 from the book \cite{Meld} because action of $\mathcal{A}$ is not non
faithfully.
Let $X=\{ x_1, x_2, \ldots ,x_n \} $ be $\mathcal{A}$-space.
If an non faithfully action by conjugation determines a shift of copies of $\mathcal{B}$ from direct product
$\mathcal{B}^n$ then we have not standard wreath product $(\mathcal{A}, X) \wr \mathcal{B}$ that is semidirect product of
$\mathcal{A}$ and $\prod \limits_{x_i \in X} \mathbb{B} $ that is $\mathcal{A}{{\ltimes }}_{\varphi } {{\left( \mathcal{B}
\right)}^{n}}$ and the following Proposition holds. Let $\mathcal{K} = ker( \mathcal{A} , X)  $ that is subgroup of
$\mathcal{A}$ that acts on $X$ as a pointwise stabiliser, that is kernel of action of $\mathcal{A}$ on $X$.

Denote the set of all the orbits of  $\mathcal{A}$ on $X$ by $\mathcal{O}$, if this set is finite then by $\mathcal{O}_f$.  Recall that the direct product indexed by infinite set consists of all infinite sequences, and the direct sum consists only of sequences with finitely many elements distinct from zero.

Denote by $Z(\tilde{\bigtriangleup}(\mathcal{B}))$ the subgroup of diagonal subgroup \cite{Dix} $Fun (X, Z(B))$ of
functions $f : X \rightarrow Z(B) $ which are constant on each orbit of action of $A$ on $X$ for unrestricted wreath
product, and denote by $Z(\bigtriangleup (\mathcal{B}^n))$ the subgroup of diagonal  $Fun(X, Z(\mathcal{B}^n))$ of
functions with the same property for restricted wreath product, where $n$ is number of non-trivial coordinates in base of
wreath product.

\begin{prop}\label{Z}
A centre of the group $ (\mathcal{A}, X) \wr \mathcal{B}$ is direct product of normal closure of centre of a diagonal of $Z
(\mathcal{B}^{n})$ i.e. ($E \times Z( \bigtriangleup( \mathcal{B}^{n}))$), trivial an element, and intersection of  $(\mathcal{K}) \times E$ with $Z(\mathcal{A})$. In other words,

\begin{equation}\label{Cent}
Z((\mathcal{A}, X) \wr \mathcal{B})= \langle (1 ; \,\, \underbrace{h,h,\ldots,h}_{n}),\,\,e,\,\, Z(\mathcal{K},X) \wr
\mathcal{E} \rangle \simeq  (Z(\mathcal{A}) \cap \mathcal{K}) \times  Z(\bigtriangleup( \mathcal{B}^{n})),
\end{equation}

 where $h \in
Z(\mathcal{B})$, $|X|=n$. 

For restricted wreath product with $n$ non-trivial coordinate:
$Z((\mathcal{A}, X) \wr \mathcal{B})= $ \\ $ \langle (1 ; \,\, \ldots,h,h,\ldots,h,\ldots),\,\,e,\,\, Z(\mathcal{K},X) \wr
\mathcal{E} \rangle \simeq  (Z(\mathcal{A}) \cap \mathcal{K}) \times  Z(\bigtriangleup(\mathcal{B}^{n})) \simeq {\underset{j \in O_f}{\bigoplus } (Z(\mathcal{A}) \cap \mathcal{K}) \times Z(\mathcal{B})}.$

In case of unrestricted wreath product we have:
$ Z((\mathcal{A}, X) \wr \mathcal{B}) = $ \\ $ \langle (1 ; \,\, \ldots,h_{-1},h_0, h_1,\ldots,h_i,h_{i+1},\ldots,
),\,\,e,\,\,
 Z(\mathcal{K},X) \wr \mathcal{E} \rangle \simeq  (Z(\mathcal{A}) \cap \mathcal{K}) \times
 Z(\tilde{\bigtriangleup}(\mathcal{B}))= {\underset{j \in O}{\prod } (Z(\mathcal{A}) \cap \mathcal{K}) \times Z(\mathcal{B})}.$
\end{prop}

\begin{proof} The elements of center subgroup have to satisfy the condition:   $f : X \rightarrow B$ such is constant on
each orbit $ \mathcal{O}_j$ of action $\mathcal{A}$ on $X$ i.e. $f(x)=b_i$ for any $x\in \mathcal{O}_j$. Also every $b_x$:
$b_x\in Z(\mathcal{B})$.
Indeed the elements of form $(1  ; \,\, \underbrace{h,h,\ldots,h}_{n})$ will not be changed by action of conjugation of any
element from $\mathcal{A}$ because any permutation elements coordinate of diagonal of $\mathcal{B}^{n}$ does not change it.
Also $h$ commutes with any element of base of $ (\mathcal{A}, X) \wr \mathcal{B}$ because $h$ from centre of
$\mathcal{B}$.
Since the action is defined by shift on finite set $X$, $|X|= n$ is not faithfully, then its kernel $ \mathcal{K} \neq E$
which confirms the proposition.
Also elements of subgroup $(\mathcal{A}, X) \wr \mathcal{E})$ belongs to $Z((\mathcal{A}, X) \wr \mathcal{B})$ iff it acts
trivial on $X$.

If we have unrestricted wreath product then we show that the center $Z(A,X \wr B)$ is the subgroup of elements of form $(a, b)$ satisfying the following conditions:
\begin{enumerate}
  \item 	$f$  is constant on each orbit of $\mathcal{A}$ on $X$, i.e. $f(x) = c_j$  for any $x \in O_j$, where $O_j \in O$, and  $c_j \in Z(\mathcal{B})$.
  \item $a \in  Z\left( \mathcal{A} \right)\cap \mathcal{K}$.
  \end{enumerate}
By the same reasons as for finite set of non-trivial elements on  $X$ applied  to $O_j$ as above It is obvious that if an element $(a,b )$ of $(A,X) \wr B$ satisfies the conditions 1) and 2), then the element $(a,b )$ belongs to the center  $Z((A,X) \ wr B)$.

On the contrary, if the element  $\left( g,\text{ }h \right)\in (A,X)\wr B$  does not satisfy one of the conditions 1) or 2) then there exists  $a\in A$ that  $ag\ne ga$  or $g$ acts non trivial on  $x_i \in X$ therefore  $\left( g,e \right)\left( a,b \right)\ne \left( a,b \right)\left( g,e \right)$, where $b$ is on  $x_i $ coordinate.  In case of violation of condition 1) there will be no commutation on the second coordinate.
\end{proof}

\vspace{5.0mm}

\begin{example}
If $\mathcal{A}= \mathbb{Z}$ then a centre $Z(( \mathbb{Z}, X)  \wr  \mathcal{B}  )= \langle ( 1 ; \,\, \underbrace{h,h,
\ldots, h}_{n}),\,\,e,\,\, n \mathbb{Z} {\ltimes} \mathcal{E} :  h \in Z( \bigtriangleup( \mathcal{B}^{n})) \rangle$.
Since the action defined by shift on finite set $X$ is not faithfully, and its kernel is isomorphic to $ n \mathbb{Z}$
because cyclic shift on $n$ coordinates is invariant on $X$.

    Generating set for commutator subgroup
	$\left( {{\mathbb{Z}}_{n}}\wr {{\mathbb{Z}}_{m}} \right)'$, where ${{\mathbb{Z}}_{n}},\,\,{{\mathbb{Z}}_{m}}$ are
presentation in additive form, is the following:
	\begin{align*}
		& {{h}_{1}}=\left( 0; 1, 0, \ldots, m-1 \right),
		\\
		&
		{{h}_{2}}=\left( 0;0,1,0, \ldots, m-1 \right),
		\\
		&
		\, \,  \,  \, \, \,  \,  \,  \, \,  \,  \, \, \,  \,  \,  \, \,  \,  \,  \, \,  \,  \,     \,  \vdots
		\\
		&
		{{h}_{n-1}}=\left( 0;0, \ldots, 1,m-1 \right).\\
	\end{align*}
	Thus, it consist of $n$ tableaux of form ${{h}_{i}}=\left({h}_{i1}, \ldots,  {h}_{im} \right)$ and relations for
coordinate of any tableau $h_i, i \in \{1, \ldots, n-1 \}$  is
$${{h}_{i1}}+ \cdots +{{h}_{in-1}}\equiv 0(\text{mod } m).$$
According to Theorem 3, for wreath product of abelian groups presented in multiplicative form, this relation has the form
$$\prod\limits_{i=1}^{n}{{{h}_{i}}{{f}_{{{i}_{{{\pi }_{a}}}}}}h_{{{i}_{{{\pi }_{a}}{{\pi }_{b}}}}}^{-1}f_{{{i}_{{{\pi
}_{a}}{{\pi }_{b}}\pi _{a}^{-1}}}}^{-1}}\left[ h,f \right]=\prod\limits_{i=1}^{n}{({{h}_{i}}{{f}_{{{i}_{{{\pi
}_{a}}}}}}h_{{{i}_{{{\pi }_{a}}{{\pi }_{b}}}}}^{-1}f_{{{i}_{{{\pi }_{a}}{{\pi }_{b}}\pi
_{a}^{-1}}}}^{-1}}\prod\limits_{j=i+1}^{i+2}{\left[ {{h}_{j}},{{f}_{{{j}_{{{\pi }_{a}}}}}} \right])}=e.$$
\end{example}

Let $(\mathbb{Z},\,{{X}_{n}})$ be infinite cyclic group acting on $n$-letters alphabet. We denote $m$-elements set by ${{X}_{m}}$.

\begin{example}
If $G\simeq (\mathbb{Z},\,{{X}_{n}})\,\wr (\mathbb{Z},\,{{X}_{m}})\wr \mathbb{Z}$ then $Z(G)\simeq n\mathbb{Z}\times m\mathbb{Z}\times \mathbb{Z}$.
\end{example}

\begin{proof}
Let us prove it applying Proposition \ref{Z}. We note that in case of action on finite set there is not difference between restricted and unrestricted wreath products. Our formula (\ref{Cent}) of centre of two groups $A$ acting on $X$ and $B$ wreath product will be used by us. This formula states that $Z(\left( A,X \right)\wr B)\simeq \left( Z(A)\cap \ker (A,X) \right)\times Z(B)$, where $\ker (A,X)$ is kernel of acting $A$ on $X$ (in our such action by shifts of $\mathbb{Z}$ on ${{X}_{n}}$ is not faithful, the kernel of action $\mathbb{Z}$ on ${{X}_{n}})$ in subgroup $n\mathbb{Z}$ because of in every $n$ shifts by cyclic the shifting point of $X_n$ returns on its initial place).

 In a case of the new group $G\simeq (\mathbb{Z},\,{{X}_{n}})\,\wr (\mathbb{Z},\,{{X}_{m}})\wr \mathbb{Z}$ we consider two groups again $\mathbb{Z}\wr \,W\,$, where $W=(\mathbb{Z},\,{{X}_{m}})\,\wr (\mathbb{Z},\,X)$. According to our formula in (\ref{Cent}) we obtain the centre $Z(W)=m\mathbb{Z}\times \mathbb{Z}$ therefore it remains to calculate the centre of the first group the kernel of its action. The kernel of action $(\mathbb{Z},\,{{X}_{n}})$ i.e. $\mathbb{Z}\,$ on ${{X}_{n}}$ is the group $n\mathbb{Z}\,$ because of the action of $\mathbb{Z}$ on $n$-elements set ${{X}_{n}}$ has the period $n$.

 Applying the formula (\ref{Cent})
  to $\mathbb{Z}\wr \,W\,$ we conclude that 
$Z\left( (\mathbb{Z},{{X}_{n}}\,)\,\wr \,W \right)\,=n\mathbb{Z}\times m\mathbb{Z}\times \mathbb{Z}$. The quotient $G/ G'$ is $\mathbb{Z} \times \mathbb{Z} \times \mathbb{Z} $ therefore the \textbf{quotient group }of $G/ G'$ by the centre $Z(G)$ is the following $\mathbb{Z}_n \times \mathbb{Z}_m \times E $. Thus, centre $Z(G)$ is subgroup of finite index in $G' = ((\mathbb{Z},\,{{X}_{n}})\,\wr (\mathbb{Z},\,{{X}_{m}})\wr \mathbb{Z}) \simeq \mathbb{Z} \times  \mathbb{Z} \times  \mathbb{Z}$.
\end{proof}
\vspace{5.0mm}
\begin{example}
If $G={{\mathbb{Z}}_{n}} \wr {{\mathbb{Z}}_{m}}$ is standard wreath product, then $d(G')=n-1$.
\end{example}

Let $G=\mathbb{Z} \wr_X \mathbb{Z}$ and $G= A \wr_X B $ be a restricted wreath product, where only $n$ non-trivial elements in coordinates of
base of wreath product which are indexed by elements from $X$, in degenerated case $\mid X \mid = n$. $Z$ acts on $X$ by
left shift. Also $A$ acts transitively from left.

\begin{remark}
The quotient group of a restricted wreath products $G= \mathbb{Z} \wr_X \mathbb{Z}$ by a commutator subgroup is isomorphic to $\mathbb{Z}\times {{\mathbb{Z}}}$. In previous conditions if $G= A \wr_X B $ then, $G/G'= A/A' \times B/B' $. If $G= {\mathbb{Z}}_n \wr {\mathbb{Z}}_m$, where $\left( m,\,\,n \right)=1$, then  $d (G/G')=1$. If $G={\mathbb{Z}} \wr {\mathbb{Z}}$ is an unrestricted regular wreath product then $G/G'\simeq {\mathbb{Z}} \times E \simeq {\mathbb{Z}} $.
\end{remark}
\begin{proof}
Consider the element of $G=A\wr_X \mathcal{B}$, where $A$ can be $\mathbb{Z}$ which acts on $X$ by left shift, then elements of commutator subgroup has form: $[e; \ldots, h_{-n}, \ldots, {{h}_{0}}, h_{1},\ldots ,{{h}_{n}}, \ldots, ]$, where ${{h}_{i}}\in B$. According to Corollary 4.9  \cite{Meld} the commutator of elements $h=[a; h_1,\ldots , h_n]$, $g=[b; g_1, \ldots , g_n]$, $g, h \in G$ satisfies the condition (\ref{CommCon}), which for case where $B$ is abelian such: $\prod\limits_{i=1}^{n}{{{h}_{i}}{{g}_{a(i)}}h_{ab(i)}^{-1}g_{ab{{a}^{-1}}(i)}^{-1}=e},$
where $g_i, \, h_i$ are non trivial coordinates from base of group,
 $a, \,b\in A$, ${{g}_{i}},{{h}_{j}}\in B$. The commutator with the shifted coordinate ${{h}_{i}}{{g}_{a(i)}}h_{ab(i)}^{-1}g_{ab{{a}^{-1}}(i)}^{-1}$ appears within the $i$-th coordinate position due to action of $A$. According to Corollary 4.9  \cite{Meld} the set of elements satisfying condition  (\ref{CommCon}) forms a commutator. Also the equivalent condition can be formulated:

\begin{equation} \label{cond1}
\prod\limits_{i=1}^{n}{{{h}_{i}}{{g}_{i}}h_{i}^{-1}g_{i}^{-1} \in \mathcal{B'}},
\end{equation}
Therefore, if $\mathcal{B}$ is abelian an element $h$ of $G$ belongs to $G'$ iff $h$ satisfy a condition: $\prod\limits_{i=1}^{n}{{h}_{i}}=e$.

 For unrestricted wreath product to show that all base of wreath product is in the commutator subgroup we choose an element $[e; \ldots, h_{-1}, h_0, h_1, \ldots ]$, where $h_i$ is variable, and form a commutator which is an arbitrary element $[e; \ldots, g_{-1}, g_0, g_1, \ldots]$ of wreath product base: $$[e; \ldots, h_{-1}, h_0, h_1, \ldots]
[\sigma;e,e, \ldots ,e][e; \ldots,
h_{-1}^{-1}, h_0^{-1}, h_1^{-1}, \ldots ][\sigma^{-1};e,e, \ldots ,e] = [e;
\ldots , g_{-1}, g_0, g_1, \ldots].$$
For convenience we present active group in the additive form. Then to previous equality holds the following equations have to be satisfied: $h_0-h_1=g_0, h_1-h_2=0, h_2-h_3=0, ....$ it implies that $h_1=h_0-1$, $h_2=h_1$, $h_3=h_2$, ... $h_i+1=h_i$. Therefore $h_i=0, \, i\geq 1$. From other side we have $h_{-1}-h_0=g_0, h_{-1}-h_{-2}=0, h_{-2}-h_{-3}=0, .... ,$ so $h_{-i}=g_0,$ for all $i<0$. That is impossible in the restricted case but possible in the unrestricted. As a corollary $G/G'\simeq Z\times Z $ for restricted case.
Thus, for unrestricted case all base of $G$ is in $G'$ as a corollary $G/G' \simeq Z\times E $.

 Thus, this group is a subdirect product of $\underbrace{B \times B \times \cdots \times B}_{n}$ with the additional condition (\ref{cond1}) where, because for any element of the subgroup of coordinates there exists a surjective homomorphism acting upon $B$, we can conclude that $G'$ must be a subdirect product.
 The commutator subgroup is the kernel of homomorphism $ \varphi: G \twoheadrightarrow G / G'$. More precisely, $G=(Z,X) \wr (Z,Y) \twoheadrightarrow G / G'  \simeq   Z/ Z'  \times Z / Z'= {{\mathbb{Z}}}\times {\mathbb{Z}}$. In case $G = A\wr \mathcal{B}$ the $ker \varphi$ has the same structure, the homomorphism $ \varphi$ maps those elements of $B^n$, as base of $G$, which satisfy $\prod\limits_{i=1}^{n}{{{h}_{i}}=e}$, i.e. the elements of $B'$ in $e$ of the group $G/G'$. Thus, $ker \varphi= G'$. To show that the properties of injectivity and surjectivity hold for this homomorphism, we chose the elements from $G$ which have the form $[e; e,\ldots e,h,e,\ldots ,e]$ that can be generator in canonical form of generating set of wreath product (\ref{can}), where $h \notin G'$, corresponding to a a specimen from the quotient group $B/B'$. Also we chose independently, an element of the form $[a; e,\ldots ,e,\ldots ,e]$ corresponding to a specimen of the quotient group $A/A'$. Therefore, we must have a one-to-one correspondence between $G/G'$ and  $A/A' \times B/B' $. In this case, we obtain ${}^{G}/{}_{G'}\simeq \left[ {}^{A}/{}_{A'}\times {}^{B}/{}_{B'} \right]$.	The basic property of homomorphism for generators in canonical form (\ref{can}) is obviously accomplished.

In the scenario when the action of $Z$ upon the $n$ elements from the set is isomorphic to the action of ${{Z}_{n}}$ elements on the set or the action of the ${{Z}_{n}}$ elements on itself.
 In case $G=Z \wr Z$ we have ${}^{G}/{}_{G'}\simeq \left[ {{\mathbb{Z}}}\times {\mathbb{Z}} \right]$.

For the group $G=Z_{n} \wr Z_{m}$ the same is true with ${}^{G}/{}_{G'}\simeq \left[ {{\mathbb{Z}}_n}\times {\mathbb{Z}_m}
\right]$
and dependently of  fact of $\left( m,n \right)=1$ or not  can admits one or two generators. For the group $G={{Z}_{n}}\wr {{Z}_{m}}$ it should be noted that the same is true. In the general case, $\underset{i=1}{\overset{n}{\mathop{\wr }}}{{\mathbb{Z}}_{{{m}_{i}}}}$ can have only one generator more than the quotient by commutator has.
\end{proof}

\subsection{Application to Geometric Groups of Diffeomorphisms Acting on the \text{M\"{o}bius} Band}
Maksymenko S. \cite{Maks} studied various different geometric objects and considered the actions of diffeomorphisms on
them. We now consider the algebraic structure and the generators for a group of such type.
	
	Let $M$ be a smooth compact connected surface, $P$ be either the real line or the circle, $f:M\to P$ be a smooth map,
	and $O(f)$ be the orbit of $f$ with respect to the right action of the group $\mathrm{Diff}(M) (\mathcal{D}(M)$) of
diffeomorphisms of $M$.
	We assume that at each critical point, the map $f$ is equivalent to a homogeneous polynomial in two variables without
multiple factors.
	Conversely, it should be noted that every group obtained in the way described will be isomorphic to $G(f)$ for some
smooth map $f:M \to P$.
	
	We will now specify the object and the construction of orbits under the action of the group diffeomorphisms. Let
$f:M\to \mathbb{R}$ be a Morse function on a connected compact surface $M$. Let $\mathcal{S}(f)$ and $O_f = \mathcal{O}(f)$
be the stabiliser and the orbit of $f$ with respect to the right action of the group of diffeomorphisms $\mathcal{D}(M)$
respectively.

Let $X_f$ denote the partition of $M$ whose elements are the connected components of  level-sets $f^{-1}(c)$ of $f$. It
should be noted here that an element $g \in X_f $ is called critical if it contains a critical point of  $f$, otherwise,
the elements is called regular. It is well known within this research domain that the factor space $M/{X_f} $, has a
natural structure of a finite graph and is entitled the Kronrod-Reeb graph.

	In our case, the diffeomorphisms act upon the \text{M\"{o}bius} band.
	Let $M$ now be a compact not orientable surface and $\omega $ be a volume from $M$ which has $h$-form of the
\text{M\"{o}bius} band. For a smooth map $f:M\to R$, denote by $S(f)$, the subgroup $\mathcal{D}(M)$ of diffeomorphisms $h$
(of $M$) which preserve $f$, i.e. those satisfying the relation  $f\circ h=f$.
	
This group is associated with ${{S}_{id}}(f)$, which is a subgroup of stabiliser elements isotopic to the identity, i.e.
${{\pi }_{1}}({{O}_{f}},f)\simeq {{\pi }_{0}}{{S}_{id}}(f)$, where the last isomorphism arises due to locally trivial
bundle.
Because there are a locally trivial bundle of homotopical groups with base ${{\pi }_{1}}\left( {{O}_{f}},\,\,f \right)$ and
layer ${{\pi }_{0}}{{S}_{id}}\left( f \right)$, this means that an exact sequence of homotopic groups and locally trivial
bundle of homotopical groups give an explanation of the isomorphism ${{\pi }_{1}}\left( {{O}_{f}},f \right)\simeq {{\pi
}_{0}}{{S}_{id}}\left( f \right)$. This locally trivial bundle of homotopical groups induce an exact sequence of homotopic
groups of that bundle. Now since the group of diffeomorphisms is infinitely dimensional, we have found the connected
components. This group is associated with the action of the group $\frac{S(f)}{{{S}_{id}}(f)}$ upon splitting into the
function level lines $f$.

In our case, the Morse function upon $M$ has two local extremes, which are the points of local maximum. Moreover, the Morse
function $f$ must have critical sets with exactly one saddle point. The lines of levels around a local maximum point of $f$
have the form of
coaxial circles, where these lines are determined by the polynomial with form $\pm ({{x}^{2}}+{{y}^{2}})+c$.
  determined by the following  homogeneous polynomial plus constant  ${{x}^{2}}-{{y}^{2}}+c$.

    Being a little more precise, we will now consider the function of Morse $f$ on $M$, which satisfies the following three
    properties:
	\begin{enumerate}
		\item $f$ is constant on the bound $M$;
		\item there are two points of maximum at a saddle point;
		\item at the two points of maximum, the values of the function are equal, i.e. at every critical point of $f$,
the germ of $f$ is $C^{\infty}$ equivalent to some homogeneous polynomial in two real variables without multiple
factors.
	\end{enumerate}
	
    Let $f:M\to R$ now be a ${{C}^{\infty }}$ Morse function.
connected was described by Maksymeko 
    We note here that since the polynomial $\pm ({{x}^{2}}+{{y}^{2}})+c$ is homogeneous and has no multiple factors, it
    follows (from the celebrated Morse Lemma) that the space of all Morse maps belongs to the space of maps $
    \mathcal{F}(\mathcal{M}, \mathcal{P})$, where $f$ here only has isolated critical points and $\mathcal{P}$ is either
    the real line $\mathbb{R}$ or the circle $S^1$.
	
    Let $\mathcal{D}(M)$  be a group of diffeomorphisms which preserve the Morse function $f$ on $M$. We know from the
    results of Maksymenko S. \cite{Maks} that $\pi_0 (D(M))\simeq\mathbb{Z}$.  Let there exist upon $M$, $n$ identical
    regions $X_i$ (critical sets) which have, for example, the form of doubles, meaning that $f$ has two critical points in
    each $X_i$ and additionally, that $X_i$ are the domains of simple connectedness.

    Consider a group $H$ of automorphisms of $M$ which are induced by the action of diffeomorphisms $h$ of a group $D(M)$
    which preserve the \text{M\"{o}bius} function $f$. In other words, the $h$ here are from the stabiliser $ S\left( f
    \right)\triangleleft D(M)$. We note that the generators with stabilisers with the right action by diffeomorphisms
    $\pi_0 S(f|_{X_i,}\partial X_i)$ are $\tau_i $. The generators of the cyclic group $Z$ which define a shift are $\rho
    $. Since the group action is continuous, this implies that the $\rho $ can realize only cyclic shifts, else one would
    change the domains of simplicity $X_i$ order.

Assume there are $n$ critical sets $X_i$ on $M$. The automorphism group $H \simeq {{\pi }_{1}}({{O}_{f}}, f)$ has exactly
two subgroups $\mathbb{Z}$ which correspond to the rotation of $M$, whose critical sets $X_i$ have not changed the order of
$X_i$ and ${{\left( \mathbb{Z} \right)}^{n}}$ denotes the subgroup of automorphisms of $n$ critical sets. Analogously to
previous investigations \cite{ Maks, Shar, SkVC}, there exists a short exact sequence $0\to {{\mathbb{Z}}^{m}}\to {{\pi
}_{1}}({{O}_{f}}, f)\to \mathbb{Z} \to 0$, where the $G$-group of automorphisms are Reeb's (Kronrod-Reeb) graph \cite{Maks}
and hence $O_f(f)$ is an orbit under action of diffeomorphism group.

The application of such an action results in a surjective epimorphism to a group  $\mathbb{Z}$, which has the left inverse
and arises as a result of splitting. The automorphism group therefore has the structure of a semi-direct product ${{\left(
\mathbb{Z} \right)}^{n}}\rtimes \ \mathbb{Z}$. This is in agreement with the work of Maksymenko S. \cite{Maks} who
considers a similar scenario but for a different group and set (surface).
Moreover, we note that this Morse function $f$ has critical sets ${{X}_{i}}$ on \text{M\"{o}bius} band $(M)$ with one
saddle point.
	
The minimal set of generators for the fundamental group ${{\pi }_{1}}({{O}_{f}},f)$ of the orbit of the function $f$ with
respect to the action of the group of diffeomorphisms of non-moving $\partial M$ is found in the next theorem.
	
We note that since the action of the group of diffeomorphisms on $n$-critical sets of $M$ have been determined and
described, the next thing to be considered is how this group has the correspondent to this action structure
	${{\pi }_{1}}({{O}_{f}},f) \simeq  \mathbb{Z}{{\ltimes  }} {{\left( \mathbb{Z} \right)}^{n}}.$
	We will denote by $H$, the fundamental group ${{\pi }_1}({O_f},f)$.
	
The first generator is $\rho $ since it realises the shift of the \text{M\"{o}bius} band. The second, $\tau $, realises the
rotation of domains ${{X}_{i}}$ of simple connectedness upon the \text{M\"{o}bius} band when passing through its twisting
point.
In other words, $\tau $ acts by the automorphism by permutation of sheets of doubles ${{X}_{i}}$ with the winding of outer
adjacency on each double ${{X}_{i}}$. Thus, we conclude that ${{\tau }_{i}}$ has infinite order.

Bounds on these domains are the lines of levels of function $f$ upon these domains, or in other words, the sets of points
with $f=const_1$.
We shall now prove that the action of the first generator $\rho $ of the group defines the homomorphism in $Aut (Z^n)$..


\begin{theorem}
The group $  H \simeq \mathbb{Z}{{\ltimes  }} {{\left( \mathbb{Z} \right)}^{n}} =\left\langle \rho ,\,\,\tau  \right\rangle
$ with defined above homomorphism in $Aut{{Z}^{n}}$ has two generators and non trivial relations
	
	$${{\rho }^{n}}\tau {{\rho }^{-n}}=\tau _{{}}^{-1}, \, \, {{\rho }^{i}}\tau {{\rho }^{-i}} {{\rho }^{j}}\tau {{\rho
}^{-j}}= {{\rho }^{j}}\tau {{\rho }^{-j}}{{\rho }^{i}}\tau {{\rho }^{-i}}, \, \, 0<i,j < n.$$
	
Also this group admits  another  presentation in generators and relations
\begin{equation}\label{relation}
\left\langle \rho ,{{\tau }_{1}}, \ldots, {{\tau }_{n}}\left| \rho {{\tau }_{i\left( \,\text{mod} \text{ } n
\right)}}{{\rho }^{-1}}={{\tau }_{i+1\left( \,\text{mod} \text{ } n \right)}} \right.,\,\,\,{{\tau }_{i}}{{\tau
}_{j}}={{\tau }_{j}}{{\tau }_{i}},\,i,\,\,j\le n \right\rangle.  		
\end{equation}
\end{theorem}

\begin{proof}
From the description above, we have that the action of the first generator of the group, $\rho $, defines the homomorphism
in the $Aut(Z^n$). There exists such a diffeomorphism from $D(M)$, called the Dehn twist, which has an infinite order since
it makes a winding of outer adjacency (refer to Dehn twist in \cite{Step}) on the doubles ${{X}_{i}}$, and it belongs to
stabiliser $S(f)$. The generator $\tau$ must therefore correspond to this diffeomorphism.
	
Let $x_i$ denote the number of domains from $X_i$. The action of the first generator $\rho $, defines the homomorphism
$\rho ({{x}_{1}}, \ldots, {{x}_{n}})={{\varphi }^{\rho }}({{x}_{1}}, \ldots, {{x}_{n}})$, where  $\varphi ({{x}_{1}},
\ldots, {{x}_{n}})=(-{{x}_{n}},{{x}_{1}}, \ldots, {{x}_{n-1}})$. It should be noted that this action could be equivalently
represented as
	$${{\varphi }^{\rho }}({{x}_{1}},\ldots, {{x}_{n}})=({{(-1)}^{\left[ \frac{ \rho +n-1}{n} \right]}}(-{{x}_{(1- \rho
)modn}}), \ldots,
	{{(-1)}^{\left[ \frac{ \rho +n-k}{n} \right]}}{{x}_{(k- \rho  )modn}}, \ldots
	,{{(-1)}^{\left[ \frac{ \rho  }{n} \right]}}{{x}_{(n- \rho  )modn}}{{x}_{n}}).$$

We extend the action of $\rho$ onto an arbitrary $\alpha \in \mathbb{Z}$. This action involves sequential shifts of $X_i$
along the orbit on $M$ defined as $\alpha $. We have
	\[{{\varphi }^{\alpha }}({{x}_{1}},\ldots,{{x}_{n}})=({{(-1)}^{\left[ \frac{ \alpha +n-1}{n} \right]}}(-{{x}_{(1-
\alpha  )modn}}),\ldots\]
	\[\,\,\ldots,{{(-1)}^{\left[ \frac{ \alpha +n-k}{n} \right]}}{{x}_{(k- \alpha  )modn}},\,\,\ldots
	\,\,,{{(-1)}^{\left[ \frac{ \alpha  }{n} \right]}}{{x}_{(n- \alpha  )modn}}{{x}_{n}}).\]

We will now consider two sets. The first set is $~\alpha = \{ 1,2,\ldots,n \} $. As an example, if $\alpha =1$, then we
have $\left[ \frac{\alpha +n-1}{n} \right]=\left[ \frac{1+n-1}{n} \right]=1$. Additionally, we find the numbers $m\in N$
such that $\left[ \frac{m+n-1}{n} \right]=1$ and the numbers that are congruent to these $m$ modulo $2 n$.

The second set is $~\alpha = \{ 0,-1,-2,\ldots,-n+1 \}$. Similarly, we note that congruence modulo $2n$ is of interest. As
an example, if $\alpha =-1$, then we have $\left[ \frac{\alpha +n-1}{n} \right]=\left[ \frac{-1+n-1}{n} \right]=0$. Hence,
$\alpha \in \mathbb{Z}$, $\alpha = l \rho$ is the number of shifts defined by $ \alpha$, while ${{\tau }_{i}}$ corresponds
to the action of automorphism by permutation with winding of outer adjacency (Dehn twist in \cite{Step}) on the doubles
${{X}_{i}}$. We conclude that ${{\tau }_{i}}$ therefore has an infinite order due to the Dehn twist.
The value of the sign of the $x_i$ indicates the presence of a rotation of the doubles or its absence.
	
The relations for the non-minimal generating set is precisely
$$\left\langle \rho ,{{\tau }_{1}}, \ldots,{{\tau }_{n}}\left| \rho {{\tau }_{i\left( \text{mod} n \right)}}{{\rho
}^{-1}}={{\tau }_{i+1\left( \text{mod} n \right)}} \right. \right\rangle.$$

This formulation yields that the relations for the minimal generating set is
	$\left\langle \rho ,\,\tau \right\rangle$ are
	$$\left\langle {{\rho }^{2n}}{{\tau }_{1}}{{\rho }^{-2n}}={{\tau }_{1}},\,\,\left| {{\tau }_{1}} \right|=\left| \rho
\right|=\infty,  {{\tau }_{1}}= \tau \right\rangle,$$
where ${{\tau }_{1}}= \tau$ and since ${{\rho }^{2n}}{{\tau }_{i}}{{\rho }^{-2n}}= \tau _{i+1}$ we transform our minimal
generating set into a canonical generating set of the $n+1$ elements given by $\left\langle \rho ,\,\tau_1, \tau_2, \ldots
, \tau_n  \right\rangle $.

It is known that the generators of the semidirect product $G \ltimes H$ may be presented in the form $(g,h)$. We now
utilise this form to say that the generators of $Z^n$ have the form of vectors $\tau_1=(h_1, e,e,??,e,)$, $\tau_2=(e,
h_2,e,??,e,),??$, $\tau_n=(e, e,??, h_n)$.
Making us of the operation of conjugation for $(e, h_1)=\tau_1 = \tau$, allows us to express the second generator of $Z^n$.
Note that this is by $(g, e)=\rho $, where $h_1$ is one of generators of  $Z^n$ and $g$ is generator of $ Z$.

$${{(g,e)}^{-1}}~\left( e,\text{ }\tau_1 \right)(g,e)=~\left( e,\text{ }\tau_2  \right).$$

Analogously, we find
$${{(g,e)}^{-1}}~\left( e,\text{ }\tau_2 \right)(g,e)=~\left( e,\text{ }\tau_3 \right),$$
and, for a general term, we find
$${{(g,e)}^{-1}}~\left( e,\text{ }\tau_{n-1} \right)(g,e)=~\left( e,\text{ }\tau_n \right).$$
	
word from $F_n$  due to conjugation by $\rho$  we can transform it to form canonocal form.

We show that there are not otherwise independent relations within the group $H$. For this group $H$, all canonical words
have the form
\begin{equation}\label{canform}
{{\rho }^{k}}\tau _{1}^{{{s}_{1}}}\tau _{2}^{{{s}_{2}}}\ldots\tau _{n}^{{{s}_{n}}}.
\end{equation}

This form follows from the form of semidirect product elements. We now prove using (\ref{relation}) and reductions of
reciprocals elements, that we may transform any finite non-trivial word of ${{F}_{n+1}}$ to the form (\ref{canform}).
Additionally, we shall prove that the set of all words which maps trivially by a surjective homomorphism, with a kernel
which is a normal closure of relations from the set $R$ from (\ref{relation}), are those coincides with trivial words in
the group $\pi(O_f, f)$. For this purpose, we prove the transformation equivalence $~{{\tau }_{i}}\rho =\rho {{\tau
}_{i+1}}$. In fact,
\begin{equation}\label{transf}
{{\tau }_{i (\text{mod} \text {} n )}} \rho =
\rho {{\rho }^{-1}}{{\tau }_{i (\text{mod} \text {} n )}}\rho =
\rho {{\tau }_{(i+1)(\text{mod} \text {} n )}}.			
\end{equation}
It should be noted that the relation ${{\tau }_{i}}{{\tau }_{j}}={{\tau }_{j}}{{\tau }_{i}}$ holds since automorphisms of
${{X}_{i}}$ and ${{X}_{j}}$ are independent of each other. Therefore, using this transformation, we can rearrange all the
$\rho $ to be in the first position in the word over the alphabet $\{\rho ,{{\tau }_{1}},\ldots,{{\tau }_{n}}\}$.

We will show that normal closure of the relations $\rho {{\tau }_{i (\text{mod} \text{} n)}}{{\rho }^{-1}}={{\tau
}_{i+1\left(\text{mod} \text{ } n \right)}}$, with ${{\tau }_{i}}{{\tau }_{j}}={{\tau }_{j}}{{\tau }_{i}}$,
determines the kernel of the surjective homomorphism $\psi$ from ${{F}_{n+1}}$ to $H$. The images of such a mapping are the
canonical words (\ref{canform}) which have the form of $H$. The form of these canonical words are determined by the
semidirect product $\mathbb{Z}{{\ltimes  }} {{\left( \mathbb{Z} \right)}^{n}}$ and its automorphisms. This mapping $\varphi
$ has the form
$$x_{{{j}_{1}}}^{{{p}_{1}}}x_{{{j}_{2}}}^{{{p}_{2}}}x_{{{j}_{3}}}^{{{p}_{3}}}\,\ldots\,\,x_{{{j}_{m}}}^{{{p}_{m}}}\mapsto
{{\rho }^{k}}\tau _{1}^{{{s}_{1}}}\tau _{2}^{{{s}_{2}}}\,\,\ldots\,\,\tau _{n}^{{{s}_{n}}},$$
where
${{x}_{{{j}_{i}}}}\in {{F}_{n+1}},$
$x_{{{j}_{1}}}^{{{p}_{1}}}x_{{{j}_{2}}}^{{{p}_{2}}}x_{{{j}_{3}}}^{{{p}_{3}}}\,\ldots\,\,x_{{{j}_{m}}}^{{{p}_{m}}} \in
{{F}_{n+1}}$ and  $\sum\limits_{i=1}^{n}{{{s}_{i}}+k}=\sum\limits_{l=1}^{m}{{{p}_{l}}}.$

For this purpose, we use the transformation equivalence $~{{\tau }_{i}}\rho =\rho {{\tau }_{i+1}}.$ Making use this
transformation allows us to therefore rearrange all $\rho $ to the first position in the word over the alphabet $\{\rho
,{{\tau }_{1}},\ldots,{{\tau }_{n}}\}$. Being a little more precise, this conversion is expressed as
$$~{{\tau }_{i (\text{mod} n)}}\rho =\rho {{\rho }^{-1}}{{\tau }_{i (\text{mod} n)}}\rho =\rho {{\tau }_{(i+1) (\text{mod}
n)}}.$$

The kernel of surjective homomorphism $\psi$ contains exactly those words that, after mapping, $\psi$ becomes the trivial
words in the group $H$ since those trivial words have the form $\rho^0 \tau_1^{0} \tau_2^{0} \ldots \tau_n^{0}$.

Note that an arbitrary word from $ker(\psi)$ may be transformed due to (\ref{transf}) into $\rho^0 \tau_1^{i_1}
\tau_2^{i_2} \ldots \tau_n^{i_n}$, where $i_k=0$ for all $k$. In fact, $ker(\psi)$ is the normal closure of the relations
(\ref{relation}) and hence it consist of the words
$\rho {{\tau }_{i\left( \,\text{mod} \text{ } n \right)}}{{\rho }^{-1}} {{\tau }^{-1}_{i+1\left( \,\text{mod} \text{ } n
\right)}}$, with $[\tau_i,\tau_{i+1}]$.

In particular, the word $\rho {{\tau }_{i\left(\text{mod} \text{ } n \right)}}{{\rho }^{-1}} {{\tau
}^{-1}_{i+1\left(\text{mod} \text{} n \right)}}$ transforms by (\ref{transf}) to $\rho {{\rho }^{-1}} {{\tau
}_{i+1\left(\text{mod} \text{} n \right)}} {{\tau }^{-1}_{i+1\left(\text{mod} \text{} n \right)}}.$ The words from the
normal closure must therefore have zero sum of powers for each generator.

In the real group $H$, with the reduced canonical words (\ref{canform}),
where all generators have infinite order, only those words with zero exponents of generators are trivial. We have therefore
found all such relations.
which concludes the proof.
\end{proof}

\vspace{5.0mm}

It should be noted that the main property of the homomorphism $\varphi $, from ${{F}_{n+1}}$ onto $H$, holds due to the
same transformation (\ref{transf}). We now consider
$$
\begin{array}{r}{\varphi(a b)=\varphi(a) \varphi(b)=\varphi\left(x_{j_{1}}^{p_{1}} x_{j_{2}}^{p_{2}} \ldots
x_{j_{m}}^{p_{m}}\right) \varphi\left(x_{i_{1}}^{q_{1}} x_{i_{2}}^{q_{2}} \ldots x_{i_{m}}^{q_{m}}\right)} \\ {=\rho^{k}
\tau_{1}^{s_{1}} \tau_{2}^{s_{2}} \ldots \tau_{n}^{s_{n}} \rho^{m} \tau_{1}^{j_{1}} \tau_{2}^{j_{2}} \ldots
\tau_{n}^{j_{n}}=\rho^{k+m} \tau_{1}^{f_{1}} \tau_{2}^{f_{2}} \ldots \tau_{n}^{f_{n}}}\end{array}
$$
Thus, the main property of the homomorphism holds. It should be noted that such a relation ${{\rho }^{2n}}{{\tau
}_{1}}{{\rho }^{-2n}}={{\tau }_{1}}$ is typical for a wreath product.

The homomorphism from the group $\mathrm{Z}$ into the group $Aut{\mathrm{Z}^{n}}$,  determining a shift of generators
$(\tau_1,\ldots,\tau_n)$ of $\mathrm{Z}^{n}$, can be equivalently presented by the matrix $\phi$. For the case $n=4$,
$\phi$ has the form
	$$\phi =\left( \begin{matrix}
	0 & 0 & 0 & -1  \\
	1 & 0 & 0 & 0  \\
	0 & 1 & 0 & 0  \\
	0 & 0 & 1 & 0  \\
	\end{matrix} \right).$$
The generators of the subgroup ${\mathrm{Z}^{n}}$ can be presented in the form of vectors. These vectors are precisely
$$(h_1, e,e,e), (e,h_2, e,e), \ldots ,(e, e, e, h_4).$$
In order to check the relation for the case $n=4$, we consider
$$\phi^4 =\left( \begin{matrix}
   -1 & 0 &  0 & 0  \\
	0 & -1& 0  & 0  \\
	0 & 0 & -1 & 0  \\
	0 & 0 & 0  & -1 \\
	\end{matrix} \right).$$
Thus,  $\phi^8=E$ and our relation ${{\rho }^{2n}}{{\tau }_{1}}{{\rho }^{-2n}}={{\tau }_{1}}$ holds.

It should be noted that the research of Maksymenko S. \cite{Maks} tells us that a group of this kind arises as a
fundamental group of the orbit ${{\pi }_{1}}({{O}_{f}}, f)$ for some Morse function $f$ which, as described above, acts
upon the \text{M\"{o}bius} band $M$.
	
\vspace{5.0mm}

Note that we have derived the relation ${{\rho }^{2n}}\tau {{\rho }^{-2n}}=\tau$. If we now multiply this from left on
${{\tau }^{-1}}$, we can equivalently express this as $${{\tau }^{-1}}{{\rho }^{2n}}\tau {{\rho }^{-2n}}=e.$$
In a similar fashion, the multiplication from the right on ${{\rho }^{2n}}$ obtains
$${{\tau }^{-1}}{{\rho }^{2n}}\tau ={{\rho }^{2n}}.$$

One such relation characterises the Bauslag-Soliter group. This is the group $G (m; k)$, which has the form $G(m;n)\text{
}=\left\langle a,b;{{a}^{-1}}{{b}^{m}}a={{b}^{k}} \right\rangle,$ where $m, k \in \mathbb{Z}$.
	Note the Bauslag-Soliter group has only one relation.

$Z$ which is elements multiple of $\rho^{2n}$ because we make a reduction by mod $2n$.

\begin{corollary}
A centre of the  group $H = \mathbb{Z}{{\ltimes  }}_{\varphi } {{\left( \mathbb{Z} \right)}^{n}}$ is a normal closure of
sets: diagonal of $\mathrm{Z}^{n}$,  trivial an element and kernel of action by conjugation that is generated by ${{\rho
}^{2n}}$ ($\langle {{\rho }^{2n}} \rangle \simeq 2n \mathrm{Z} $). In other words,
$$Z(H)=\langle ( 1 ; \,\, \underbrace{h,h,\ldots,h}_{n}),\,\,e,\,\,  2n \mathrm{Z} \ltimes \mathrm{ E }  \rangle,$$ where
$h,\,\,g\in \mathrm{Z}$.
Thus, $Z(H) \simeq 2n\mathbb{Z} \times \mathbb{Z}. $
Since the action is defined by conjugation and relation ${{\rho }^{2n}}{{\tau }_{i}}{{\rho }^{-2n}}={{\tau }_{i}}$ holds
then the element ${{\rho }^{2n}}$ commutates with every ${{\tau }_{i}}$. So subgroup stabilise all $x_i$ of $Z$-space $M$.
Other words subgroup $\langle{{\rho }^{2n}}\rangle$ belongs to kernel of action $\phi$. Besides the element $(1  ; \,\,
\underbrace{h,h,\ldots,h}_{n})$ will not be changed by action of conjugation of any element from $H$ because any
permutation elements coordinate of diagonal of $\mathrm{Z}^{n}$ does not change it.
\end{corollary}

We generalize the result of Meldrum J. \cite{Meld} because we consider not only the permutation wreath product groups,
but the group $\mathcal{A}$ does not have to act upon the set $X$ faithfully, hence $(\mathcal{A}, X)\wr \mathcal{B}$ is not regular wreath product.
Recall that an action is said to be faithful if for every $g\in G$, there exists $x$ from $G$-space $X$ such that $x^g \neq
x$. We consider wreath products with no regular actions of the active group $\mathbb{Z}$. 

Let $X=\{ x_1, x_2, \ldots ,x_n \} $ be the $\mathbb{Z}$-space. If an action by conjugation determines a shift of the
copies of $\mathbb{Z}$ from the direct product $\mathbb{Z}^n$ then, we have not found a standard wreath product
$(\mathbb{Z}, X) \wr \mathbb{Z}$ which is a semidirect product of $\mathbb{Z}$ and $\prod \limits_{x_i \in X} \mathbb{Z} $,
i.e. $\mathbb{Z}{{\ltimes}}_{\phi } {{\left( \mathbb{Z} \right)}^{n}}$. Thus, we observe the following corollary holds.

\begin{corollary}
The centre of the group $ \mathbb{Z}{{\ltimes  }}_{\phi } {{\left( \mathbb{Z} \right)}^{n}} \simeq (\mathbb{Z}, X) \wr
\mathbb{Z}$ consists of normal closure of the diagonal of $\mathrm{Z}^{n}$, a trivial element and the kernel of action by
conjugation, i.e. $n \mathrm{Z} $. In other words, $$Z(H)=\langle \, 1 ; \,\, \underbrace{h,h,\ldots,h}_{n}),\,\,e,\,\,
(n \mathbb{Z}, X) \wr \mathbb{E} \rangle  \simeq n\mathbb{Z} \times \mathbb{Z} ,$$ where $h,\,\,g\in \mathrm{Z}$, $Z(H)
\simeq n\mathbb{Z} \times \mathbb{Z} $.
\end{corollary}

\begin{proof}
The proof follows immediately from Corollary 1 by utilising the kernel of action of $\phi$. The stabiliser of such an
action over the $\mathbb{Z}$-space $X=\{ x_1, x_2, \ldots , x_n \} $ is the subgroup $n\texttt{Z}$. Additionally, the
kernel of this action has elements from the diagonal of $\mathbb{Z}^n$.

It should be noted, if we have $G= \mathbb{Z} \underset{X_m} {\wr} \mathbb{Z} \underset{X_n}{\wr} \mathbb{Z},$ where $| X_m
| =m$ and $| X_n | = n$, then the action is defined by the shift upon finite set $X_n$. In this case, we find that $|X|= n$
is not faithful and its kernel is also isomorphic to $n \mathbb{Z}$ since the cyclic shift on the $n$ coordinates is
invariant on $X$. Note that the action is defined by the shift on the finite set $X_m$ is not faithful and its kernel is
isomorphic to $ m \mathbb{Z}$. Additionally, within this kernel of action is the elements from the diagonal of
$\mathbb{Z}^{nm}$ which are isomorphic to $\mathbb{Z}$. Thus, its centre is $Z(G) \simeq n \mathbb{Z} \times m \mathbb{Z}
\times \mathbb{Z}$ which concludes the proof.
\end{proof}


\begin{remark}
The centre of a group of the form $\mathbb{Z}{{\ltimes  }}_{\phi } {{\left( \mathcal{B} \right)}^{n}} \simeq (\mathbb{Z},
X) \wr \mathcal{B}$ generates, by normal closure of: centre of diagonal of $\mathcal{B}^{n}$,  trivial an element, and $n
\mathrm{Z}  {\underset{X} {\wr}}  \mathcal{E}$.
\end{remark} 

\section{Conclusion}
The minimal generating set and the structure of the group $\pi_{0}S_{id}(f)$ of the orbit one Morse function have been
investigated. The minimal generating set for wreath-cyclic groups have been constructed.











\begin{thebibliography}{10}

\bibitem{Grig}
Laurent Bartholdi, Rostislav~I Grigorchuk, and Zoran {\v{S}}uni.
\newblock Branch groups.
\newblock In {\em Handbook of algebra}, volume~3, pages 989--1112. Elsevier,
  2003.

\bibitem{Meld}
John~DP Meldrum.
\newblock {\em Wreath products of groups and semigroups}, volume~74.
\newblock CRC Press, 1995.


\bibitem{Bon}
Ievgen~V Bondarenko.
\newblock Finite generation of iterated wreath products.
\newblock {\em Archiv der Mathematik}, 95(4):301--308, 2010.

 \bibitem{Wor} A. Woryna, The rank and generating set for iterated wreath products of cyclic groups, Communications in Algebra, 39 (7) (2011), 2622-2631

\bibitem{Dix}
John~D Dixon and Brian Mortimer.
\newblock {\em Permutation groups}, volume 163.
\newblock Springer Science \& Business Media, 1996.

\bibitem{Dm}
Yu~V Dmitruk and VI~Sushchanskii.
\newblock Structure of sylow 2-subgroups of the alternating groups and
  normalizers of sylow subgroups in the symmetric and alternating groups.
\newblock {\em Ukrainian Mathematical Journal}, 33(3):235--241, 1981.

\bibitem{Step}
Stephen~P Humphries.
\newblock Generators for the mapping class group.
\newblock In {\em Topology of low-dimensional manifolds}, pages 44--47.
  Springer, 1979.

\bibitem{Isacs}
I.~Martin Isaacs.
\newblock Commutators and the commutator subgroup.
\newblock {\em The American Mathematical Monthly}, 84(9):720--722, 1977.

\bibitem{Kal}
L{\'e}o Kaloujnine.
\newblock Sur les $p$-groupes de sylow du groupe sym\'etrique du degr\'e $p^m$.
\newblock {\em Comptes Rendus de l'Acad{\'e}mie des sciences}, 221:222--224,
  1945.

\bibitem{Lav}
Yaroslav Lavrenyuk.
\newblock On the finite state automorphism group of a rooted tree.
\newblock {\em Algebra and Discrete Mathematics}, pages 79--87, 2002.

\bibitem{Luc}
Andrea Lucchini.
\newblock Generating wreath products and their augmentation ideals.
\newblock {\em Rendiconti del Seminario Matematico della Universit{\`a} di
  Padova}, 98:67--87, 1997.

\bibitem{Maks}
Sergiy Maksymenko.
\newblock Deformations of functions on surfaces by isotopic to the identity
  diffeomorphisms.
\newblock {\em arXiv preprint arXiv:1311.3347}, 2013.

\bibitem{Agn}
	R. Skuratovskii,
\newblock The Derived Subgroups of Sylow 2-Subgroups of the Alternating Group and  Commutator Width of Wreath Product of Groups. \newblock {\em Mathematics, Basel, Switzerland}, (2020),  8(4), pp. 1-19.

\bibitem{Mur}
Alexey Muranov.
\newblock Finitely generated infinite simple groups of infinite commutator
  width.
\newblock {\em International Journal of Algebra and Computation},
  17(03):607--659, 2007.

\bibitem{Ne}
Volodymyr Nekrashevych.
\newblock {\em Self-similar groups}, volume 117.
\newblock American Mathematical Society, 2005.

\bibitem{nikolov}
Nikolay Nikolov.
\newblock On the commutator width of perfect groups.
\newblock {\em Bulletin of the London Mathematical Society}, 36(1):30--36,
  2004.

\bibitem{Shar}
Vladimir Sharko.
\newblock Smooth and topological equivalence of functions on surfaces.
\newblock {\em Ukrainian mathematical journal}, 55(5):832--846, 2003.

\bibitem{SkKib}
Ruslan Skuratovskii.
\newblock Corepresentation of a sylow p-subgroup of a group s n.
\newblock {\em Cybernetics and systems analysis}, 45(1):25--37, 2009.

\bibitem{SkVC}
Ruslan Skuratovskii.
\newblock  Minimal generating set and structure of wreath product of cyclic groups, coomutator of wreath product and the fundamental group of Morse function $pi_1 O(f)$
\newblock [Source: Arxiv.org], access mode: https://arxiv.org/abs/1901.00061v1.pdf


\bibitem{SkComm}
Ruslan Skuratovskii.
\newblock The commutator and centralizer description of sylow 2-subgroups of
  alternating and symmetric groups.
\newblock {\em arXiv preprint arXiv:1712.01401}, 2017.



\bibitem{SkAr}
Ruslan Skuratovskii.
\newblock The commutator and centralizer of sylow subgroups of alternating and
  symmetric groups, its minimal generating set.
\newblock In {\em International Scientific Conference. Algebraic and geometric
  methods of analysis}, page~58, 2018.

\bibitem{SkAr2} Skuratovskii~R.~V., The commutator subgroup of Sylow 2-subgroups of
alternating group, commutator width of wreath product [Source: Arxiv.org], access mode:  https://arxiv.org/pdf/1903.08765.pdf.

\bibitem{SkArM} Skuratovskii~R.~V., Minimal generating set and structure of wreath product of cyclic groups, coomutator of wreath product and the fundamental group of Morse function $pi_1 O(f)$ [Source: Arxiv.org], access mode: https://arxiv.org/abs/1901.00061v1.pdf

\bibitem{SkArM3} Skuratovskii~R.~V., Minimal generating set and structure of wreath product of cyclic groups, coomutator of wreath product and the fundamental group of Morse function $pi_1 O(f)$ [Source: Arxiv.org], access mode:
https://arxiv.org/pdf/1901.00061v12.pdf




\bibitem{SkMal}
Ruslan Skuratovskii.
\newblock Minimal generating sets of cyclic groups wreath product (in russian).
\newblock In {\em International Conference, Mal'tsev meetting}, page 118, 2018.

\bibitem{ShNorm}
Vitaly~Ivanovich Sushchansky.
\newblock Normal structure of the isometric metric group spaces of p-adic
  integers.
\newblock {\em Algebraic structures and their application. Kiev}, pages
  113--121, 1988.

\bibitem{W}
James Wiegold.
\newblock Growth sequences of finite groups.
\newblock {\em Journal of the Australian Mathematical Society}, 17(2):133--141,
  1974.

\bibitem{SkRendi} R. V. Skuratovskii, A. Williams "Irreducible bases and subgroups of a wreath product in applying to diffeomorphism groups acting on the Mobius band", \textbf{2021}.
    Rendiconti del Circolo Matematico di Palermo Series 2, 70(2), 721-739. https://doi.org/10.1007/s12215-020-00514-5


\end{thebibliography}



\end{document}